\documentclass[a4paper,12pt]{article}

\usepackage[textwidth=125mm, textheight=195mm]{geometry}
\usepackage[english]{babel}
\usepackage{graphicx}
\usepackage{amsmath}
\usepackage{amsfonts}
\usepackage{amsthm}
\usepackage{epstopdf}
\usepackage{color}

\begin{document}

\newcommand{\wk}{\mbox{$\,<$\hspace{-5pt}\footnotesize )$\,$}}

\numberwithin{equation}{section}
\newtheorem{teo}{Theorem}
\newtheorem{lemma}{Lemma}

\newtheorem{coro}{Corollary}
\newtheorem{prop}{Proposition}
\theoremstyle{definition}
\newtheorem{definition}{Definition}
\theoremstyle{remark}
\newtheorem{remark}{Remark}

\newtheorem{scho}{Scholium}
\newtheorem{open}{Question}
\newtheorem{example}{Example}
\numberwithin{example}{section}
\numberwithin{lemma}{section}
\numberwithin{prop}{section}
\numberwithin{teo}{section}
\numberwithin{definition}{section}
\numberwithin{coro}{section}
\numberwithin{figure}{section}
\numberwithin{remark}{section}
\numberwithin{scho}{section}

\bibliographystyle{abbrv}

\title{On curvature of surfaces immersed in normed spaces}
\date{}
\author{Vitor Balestro\footnote{Corresponding author}  \\ CEFET/RJ Campus Nova Friburgo \\ 28635000 Nova Friburgo \\ Brazil \\ vitorbalestro@gmail.com \and Horst Martini \\ Fakult\"{a}t f\"{u}r Mathematik \\ Technische Universit\"{a}t Chemnitz \\ 09107 Chemnitz\\ Germany \\ martini@mathematik.tu-chemnitz.de \and  Ralph Teixeira \\ Instituto de Matem\'{a}tica e Estat\'{i}stica  \\ Universidade Federal Fluminense \\24210201 Niter\'{o}i\\ Brazil \\ ralph@mat.uff.br}

\maketitle

\begin{abstract} The normal map given by Birkhoff orthogonality yields extensions of principal, Gaussian and mean curvatures to surfaces immersed in three-dimensional spaces whose geometry is given by an arbitrary norm and which are also called Minkowski spaces. We obtain characterizations of the Minkowski Gaussian curvature in terms of surface areas, and respective generalizations of the classical theorems of Huber, Willmore, Alexandrov, and Bertrand-Diguet-Puiseux are derived. A generalization of Weyl's formula for the volume of tubes and some estimates for volumes and areas in terms of curvature are obtained, and in addition we discuss also two-dimensional subcases of the results in more detail.

\end{abstract}

\noindent\textbf{Keywords}: Alexandrov's theorem, Birkhoff-Gauss map, Finsler manifold, Minkowski curvature, normed space, Weyl's tube formula.

\bigskip

\noindent\textbf{MSC 2010:} 53A35, 53A15, 53A10, 58B20, 52A15, 52A21, 46B20

\section{Introduction}

We are concerned with the extension of the classical differential geometry of surfaces to the context where the ambient space is equipped with a norm not necessarily induced by an inner product. Thus we refer to the geometry of real, finite dimensional, normed spaces, usually also called \emph{Minkowski geometry} (see the monograph \cite{Tho}). A surface immersed in such a space becomes a Finsler manifold with the metric induced from the ambient norm, and we also can define an analogue of the Gauss map given by the orthogonality concept which is commonly used in Minkowski geometry (namely, that of \emph{Birkhoff orthogonality}). This is extensively studied in the papers \cite{Ba-Ma-Tei1}, \cite{Ba-Ma-Tei2} and \cite{Ba-Ma-Tei3}, where the fundaments of this theory are established. Besides \cite{Tho}, basic references on Minkowski geometry are the surveys \cite{martini2} and \cite{martini1}. Our objective in the present paper is to prove some results that extend geometric interpretations of the classical curvature concepts to the Minkowski context. We start by briefly describing the basic concepts of the theory, referring the reader also to the mentioned papers. In a comparing sense we mention also \cite{Ba-Ma-Sho}, where curvature concepts of curves in Minkowski planes are discussed.

Let $||\cdot||$ be a norm in $\mathbb{R}^3$ which is smooth and strictly convex. Its \emph{unit ball} is the set $B:=\{x \in \mathbb{R}^3:||x|| \leq 1\}$, and the boundary $\partial B:=\{x \in \mathbb{R}^3:||x|| = 1\}$ of $B$ is called the \emph{unit sphere} of the Minkowski space $(\mathbb{R}^3,||\cdot||)$. It is clear that $\partial B$ is a compact surface without boundary, and througout the text, we will always assume that its (usual) Gaussian curvature never vanishes. We denote by $\langle\cdot,\cdot\rangle:\mathbb{R}^3\times \mathbb{R}^3\rightarrow\mathbb{R}$ the usual inner product of $\mathbb{R}^3$, and the norm induced by that is called \emph{Euclidean}. The unit ball and unit sphere of the Euclidean norm are denoted by $B_e$ and $\partial B_e$, respectively. We also define the map $u:\partial B_e\rightarrow \partial B$ to be the inverse of the (usual) Gauss map of $\partial B$ (outward pointing, say).

The norm $||\cdot||$ induces an orthogonality relation known as \emph{Birkhoff orthogonality}. We say that a vector $v \in \mathbb{R}^3$ is \emph{Birkhoff orthogonal} to a plane $H \subseteq \mathbb{R}^3$ whenever $||v|| \leq ||v+tw||$ for any $w \in H$ and $t \in \mathbb{R}$. Geometrically this means that, if $v \neq 0$, then the plane $H$ supports the unit ball at $v/||v||$ (for more on orthogonality concepts in Minkowski geometry we refer the reader to \cite{alonso} and references therein). Given a (orientable) surface immersion $f:M\rightarrow(\mathbb{R}^3,||\cdot||)$, Birkhoff orthogonality defines a map by associating each point $p \in M$ to a unit vector $\eta(p)$ which is Birkhoff orthogonal to the tangent plane $T_pM$ (identified with $f_*(T_pM)$, of course). This map $\eta:M\rightarrow \partial B$ is called the \emph{Birkhoff-Gauss map} of $M$. In the language of affine differential geometry, one can show that this is an \emph{equiaffine normal vector field} on $M$ (see \cite{nomizu}). The \emph{Minkowski Gaussian curvature} and the \emph{Minkowski mean curvature} of $M$ at $p \in M$ are defined as
\begin{align*} K(p) := \mathrm{det}\left(d\eta_p\right) \  \ \mathrm{and} \\
H(p) := \frac{1}{2}\mathrm{tr}(d\eta_p),
\end{align*} 
respectively, where $\mathrm{det}$ and $\mathrm{tr}$ denote the usual determinant and trace. Each differential map $d\eta_p$ is self-adjoint, and hence the Minkowski Gaussian curvature equals the product of its eigenvalues, which we call the \emph{principal curvatures}. 

\section{Area and curvature}

The usual determinant in $\mathbb{R}^3$ induces an area element in $M$ as
\begin{align}\label{areaform} \omega(X,Y) := \mathrm{det}(X,Y,\eta),
\end{align}
for each $X,Y \in TM$. With that area element, the \emph{area} of an open bounded subset $D \subseteq M$ is given as
\begin{align*} \lambda_M(D) := \int_{D}\omega.
\end{align*}

In the next theorem, we characterize the Minkowski Gaussian curvature at a point $p \in M$ as the limit ratio between the area of the image of a small region around $p$ under the Birkhoff-Gauss map and the area of the region itself. We adopt the convention that this area is negative in a subset where the Minkowski Gaussian curvature is negative. 

\begin{teo}\label{main} Let $p \in M$ be a point where $K(p) \neq 0$, and let $U\subseteq M$ be a connected neighborhood of $p$ where the sign of $K$ does not change. Then
\begin{align}\label{gaussian} K(p) = \lim_{D\rightarrow p}\frac{\lambda_{\partial B}(\eta(D))}{\lambda_M(D)},
\end{align}
where $D\subseteq U$ and $\eta(D) \subseteq \partial B$ denotes the image of $D$ under the Birkhoff-Gauss map.
\end{teo}
\begin{proof} Let $\varphi(u,v):V\rightarrow D$ be a local parametrization, where $V$ is an open disk, say. The area of $D$ writes
\begin{align*} \lambda_M(D) = \int_V\omega(\varphi_u,\varphi_v) \ dudv,
\end{align*}
where $\varphi_u$ and $\varphi_v$ are the coordinate vector fields. Since $K(p) \neq 0$, by the inverse function theorem one can take $D$ small enough such that the restriction $\eta|_D$ is a diffeomorphism onto its image. In that case, the map $\eta\circ\varphi:V\rightarrow \partial B$ becomes a local parametrization of $\eta(D)$. The area of this region is calculated by
\begin{align}\label{areaball} \lambda_{\partial B}(\eta(D)) = \int_V\omega(d\eta(\varphi_u),d\eta(\varphi_v)) \ dudv = \int_V\mathrm{det}(d\eta)\cdot\omega(\varphi_u,\varphi_v) \ dudv.
\end{align}
Denoting by $\lambda(V)$ the usual area of $V$, we get from the mean value theorem for integrals that for each region $D$ there exists a point $q_1 \in D$ such that  
\begin{align*} \omega_{q_1}(\varphi_u,\varphi_v) = \frac{1}{\lambda(V)} \int_{V}\omega(\varphi_u,\varphi_v) \ dudv = \frac{\lambda_M(D)}{\lambda(V)},
\end{align*}
and the same holds for the other integral, for some point $q_2 \in D$. As $D \rightarrow p$, we get that $q_1,q_2 \rightarrow p$, and hence continuity yields
\begin{align*} \lim_{D\rightarrow p}\frac{\lambda_{\partial B}(\eta(D))}{\lambda_M(D)} = \lim_{D\rightarrow p}\frac{\frac{\lambda_{\partial B}(\eta(D))}{\lambda(V)}}{\frac{\lambda_M(D)}{\lambda(V)}} = \lim_{D\rightarrow p} \frac{\mathrm{det}\left(d\eta_{q_2}\right)\cdot\omega_{q_2}(\varphi_u,\varphi_v)}{\omega_{q_1}(\varphi_u,\varphi_v)} = K(p),
\end{align*}
completing the proof.

\end{proof}

\begin{remark}\label{curve} This theorem shows that the Minkowski Gaussian curvature is, in some sense, an extension of the concept of \emph{circular curvature} of a plane curve, cf. \cite{Ba-Ma-Sho}. Let $(\mathbb{R}^2,||\cdot||)$ be a normed plane with unit sphere $\partial B$ (which, in the two-dimensional case, is also called \emph{unit circle}), and assume that this is a smooth regular curve. Let $\varphi(t)$ be a parametrization by arc-length of the unit circle, and let $\gamma(s)$ be a smooth regular curve parametrized by arc-length. We choose a smooth function $t(s)$ such that 
\begin{align*} \gamma'(s) = \frac{d\varphi}{dt}(t(s)),
\end{align*}
and the \emph{circular curvature} of $\gamma$ at $\gamma(s)$ is defined as $k_c(s) := t'(s)$, see \cite{Ba-Ma-Sho} for further details. It is also worth mentioning that if $k_c(s) > 0$, then $1/k_c(s)$ is the radius of an osculating Minkowski circle attached to $\gamma$ at $\gamma(s)$.  The analogue of the Birkhoff-Gauss map for the curve $\gamma$ is clearly the map $s \mapsto \varphi(t(s))$, and hence $t(\bar{s})-t(s)$ equals the length of the image, under the Birkhoff-Gauss map, of an arc of $\gamma$ whose length is $\bar{s}-s$. Analogously to (\ref{gaussian}) we get
\begin{align*} k_c(s) = t'(s) = \lim_{\bar{s}\rightarrow s}\frac{t(\bar{s})-t(s)}{\bar{s}-s}.
\end{align*}
\end{remark}

A \emph{homothety} of the space is a map $F:\mathbb{R}^3\rightarrow\mathbb{R}^3$ given as $F(p) = cp$ for some constant $c > 0$. As a consequence of the previous theorem we will describe what happens with the Minkowski Gaussian curvature of an immersed surface under a homothety. 

\begin{coro} Let $M$ be an immersed surface in $(\mathbb{R}^3,||\cdot||)$, and $c > 0$ be a constant. For each $p \in M$, the Minkowski Gaussian curvature $\bar{K}(F(p))$ of the image $\bar{M}$ of $M$ by the homothety $F(p) = cp$ at $F(p)$ is given as
\begin{align*} \bar{K}(F(p)) = \frac{K(p)}{c^2},
\end{align*}
where $K$ denotes the Minkowski Gaussian curvature of $M$ at $p$. 
\end{coro}
\begin{proof} Let $p \in M$, and $\varphi(u,v):V\rightarrow D$ be a local parametrization of a neighborhood $D \subseteq M$ of $p$. Then, $c\cdot\varphi(u,v):V\rightarrow D$ is a local parametrization of $F(M)$ around $F(p)$, and
\begin{align*} \lambda_{\bar{M}}(\bar{D}) = \int_V\omega(c\varphi_u,c\varphi_v) \ dudv = c^2\lambda_M(D).
\end{align*}
On the other hand, it is clear that the normal vector to $M$ at any $q \in D$ is the same as the normal vector to $\bar{M}$ at $F(q) \in \bar{D} = F(D)$. Therefore,
\begin{align*} \lambda_{\partial B}(\eta(\bar{D})) = \lambda_{\partial B}(\eta(D)).
\end{align*}
Now we just calculate
\begin{align*} \bar{K}(F(p)) = \lim_{\bar{D}\rightarrow F(p)}\frac{\lambda_{\partial B}(\eta(\bar{D}))}{\lambda_{\bar{M}}(D)} = \lim_{D\rightarrow p}\frac{\lambda_{\partial B}(\eta(D))}{c^2\lambda_M(D)} =\frac{1}{c^2}\cdot\lim_{D\rightarrow p}\frac{\lambda_{\partial B}(\eta(D))}{\lambda_M(D)}= \frac{K(p)}{c^2},
\end{align*}
where we use the clear fact that $\bar{D}\rightarrow F(p)$ if and only if $D\rightarrow p$. 

\end{proof}

\begin{remark} This can be also obtained from the fact that the principal curvatures of $M$ get divided by $c$ under the homothety $p \mapsto cp$. This fact comes immediately from the characterization of the principal curvatures at a given point as the maximum and minimum of the normal curvature at that point (see \cite{Ba-Ma-Tei1}). 
\end{remark}

Given a surface $M\subseteq\mathbb{R}^3$ (recall that we are identifying $f(M)$ with $M$), a \emph{parallel surface} of $M$ is a surface $\bar{M} := \{p+c\eta(p):p \in M\}$ for some constant $c \in \mathbb{R}$. As a further consequence of Theorem \ref{main} we get a formula for the Minkowski Gaussian curvature of a parallel surface in a regular point (a parallel surface can have singular points). 

\begin{teo}\label{parallel} Let $M$ be an immersed surface with Minkowski Gaussian curvature $K$ and Minkowski mean curvature $H$. For a given constant $c \in \mathbb{R}$, let $\bar{M}$ be the parallel surface as defined above. Then its Minkowski Gaussian curvature is given by the formula
\begin{align*} \bar{K}(p+c\eta(p)) = \frac{K(p)}{c^2K(p)+2cH(p)+1},
\end{align*}
for each $p \in M$ such that $p+c\eta(p)$ is a regular point of $\bar{M}$. 
\end{teo} 
\begin{proof} As in the proof of Theorem \ref{main}, let $\varphi(u,v):V\rightarrow D$ be a local parametrization of a neighborhood of $p \in M$. If $p$ is not umbilic, then we can assume that $\varphi$ is a parametrization whose coordinate curves are curvature lines; that is, we can assume that
\begin{align*} d\eta(\varphi_u) = \lambda_1\varphi_u \ \ \mathrm{and} \\
d\eta(\varphi_v) = \lambda_2\varphi_v,
\end{align*}
where $\lambda_1$ and $\lambda_2$ are the principal curvatures of $M$. The map $\psi(u,v):V\rightarrow \bar{D} :=\{q+c\eta(q):q \in D\}$ defined as $\psi(u,v) = \varphi(u,v)+c\eta(u,v)$ is a local parametrization of the neighborhood $\bar{D}$ of $p+c\eta(p) \in \bar{M}$. We clearly have
\begin{align*} \psi_u = (1+c\lambda_1)\varphi_u \ \ \mathrm{and}\\
\psi_v = (1+c\lambda_2)\varphi_v,
\end{align*}
and also, that the Birkhoff normal to $M$ at $p$ is the same as the Birkhoff normal to $\bar{M}$ at $p+c\eta(p)$. Hence the area of $\bar{D}$ is calculated as
\begin{align*} \lambda_{\bar{M}}(\bar{D}) = \int_V\omega(\psi_u,\psi_v) \ dudv = \int_V(1+c\lambda_1)(1+c\lambda_2)\omega(\varphi_u,\varphi_v) \ dudv,
\end{align*}
and from the mean value theorem for integrals we get that
\begin{align*} \frac{\lambda_{\bar{M}}(\bar{D})}{\lambda(V)} = (1+c\lambda_1(q_1))(1+c\lambda_2(q_1))\omega_{q_1}(\varphi_u,\varphi_v),
\end{align*}
for some $q_1 \in V$. From equality (\ref{areaball}), and repeating the argument of the mean value theorem for the ratio $\frac{\lambda_{\partial B}(\eta(D))}{\lambda(V)}$, we calculate the Minkowski Gaussian curvature of $\bar{M}$ at $p+c\eta(p)$ as
\begin{align*} \bar{K}(p+c\eta(p)) = \lim_{D\rightarrow p}\frac{\lambda_{\partial B}(\eta(\bar{D}))}{\lambda_{\bar{M}}(\bar{D})} = \lim_{D\rightarrow p}\frac{\frac{\lambda_{\partial B}(\eta(D))}{\lambda(V)}}{\frac{\lambda_{\bar{M}}(\bar{D})}{\lambda(V)}} = \lim_{D\rightarrow p}\frac{K(q_2)\omega_{q_2}(\varphi_u,\varphi_v)}{(1+c\lambda_1(q_1))(1+c\lambda_2(q_1))\omega_{q_1}(\varphi_u,\varphi_v)} = \\ = \frac{K(p)}{(1+c\lambda_1(p))(1+c\lambda_2(p))} = \frac{K(p)}{c^2K(p)+2cH(p)+1},
\end{align*}
where we again use the fact that $q_1,q_2 \rightarrow p$ as $D \rightarrow p$. If $p$ is an isolated umbilic point, then we get the result using a continuity argument. As it is proved in \cite{Ba-Ma-Tei1}, open sets all whose points are umbilic must be contained in Minkowski spheres. 

\end{proof}

\section{Weyl's tube formula and intrinsic volumes}

Inspired by Theorem \ref{parallel}, we devote this section to obtaining an analogue of Weyl's tube formula, which characterizes the volume of the set of the points $\varepsilon$-next to a surface $M$ as a polynomial of degree $3$ in the variable $\varepsilon$ ($>0$, say). In what follows, the volume in $\mathbb{R}^3$ is given by the usual determinant. 

\begin{teo}[Weyl's tube formula]\label{weyl} Let $M$ be a surface in $(\mathbb{R}^3,||\cdot||)$, and let $M_{\varepsilon}$ be the $\varepsilon$\emph{-tube} defined as
\begin{align*} M_{\varepsilon} := \{z \in \mathbb{R}^3:\mathrm{dist}(z,M) \leq \varepsilon\},
\end{align*}
where $\mathrm{dist}(z,M) = \inf\{||z-p||: p\in M\}$. For sufficiently small $\varepsilon > 0$, the volume of $M_{\varepsilon}$ is given by the formula
\begin{align}\label{weyl} \mathrm{vol}(M_{\varepsilon}) = 2\varepsilon\lambda_M(M) + \frac{2\varepsilon^3}{3}\int K \omega,
\end{align}
where $\omega$ is the area element induced in $M$ as in \emph{(\ref{areaform})}. 
\end{teo}
\begin{proof} The idea of the proof is to see $M_{\varepsilon}$ as a family of parallel surfaces. We assume that $D \subseteq M$ is a coordinate neighborhood which can be parametrized as in Theorem \ref{parallel}. That is, we assume that $\phi(u,v):V\rightarrow D$ is such that 
\begin{align*} d\eta(\phi_u) = \lambda_1\phi_u \ \  \mathrm{and} \\
d\eta(\phi_v) = \lambda_2\phi_v,
\end{align*} 
where $\lambda_1$ and $\lambda_2$ are the principal curvatures of $M$. For the general case, we use continuity arguments and partitions of the unity. Given such a parametrization of a neighborhood $D$, for sufficiently small $\varepsilon > 0$ one can parametrize $D_{\varepsilon}:=\{z \in \mathbb{R}^3:\mathrm{dist}(z,D) \leq \varepsilon\}$ as
\begin{align*} \varphi(u,v,t) = \phi(u,v) + t\eta(u,v),
\end{align*}
for $(u,v,t) \in V\times(-\varepsilon,\varepsilon)$. Denote by $K$ and $H$ the Minkowski Gaussian and mean curvatures of $M$ at $(u,v)$, and write $(u,v,t) = x$ for simplicity. Hence the volume of $D_{\varepsilon}$ is calculated as
\begin{align*} \mathrm{vol}(D_{\varepsilon}) = \int_{V\times(-\varepsilon,\varepsilon)}\mathrm{det}(\varphi_u,\varphi_v,\varphi_t) \ dx = \int_{V\times(-\varepsilon,\varepsilon)}(1+t\lambda_1)(1+t\lambda_2)\mathrm{det}(\phi_u,\phi_v,\eta) \ dx = \\ = \int_{V\times(-\varepsilon,\varepsilon)}(1+2tH+t^2K)\omega(\phi_u,\phi_v) \ dx = \\ = 2\varepsilon \int_D \omega + \left(\int_{-\varepsilon}^{\varepsilon}2t \ dt\right) \cdot \left(\int_VH \omega\right) + \frac{2\varepsilon^3}{3}\int_{D}K \omega = \\ = 2\varepsilon\lambda_M(D) + \frac{2\varepsilon^3}{3}\int_{D}K \omega,
\end{align*}
where we used Fubini's theorem. Of course, using partitions of the unity we get equality for $M_{\varepsilon}$ (as we mentioned previously). 

\end{proof}

\begin{remark} In the previous theorem, and also in Theorem \ref{parallel}, the choice of a parametrization whose coordinate curves are curvature lines is not necessary. We use this only to make the calculations easier. This comes with the disadvantage of needing a continuity argument to deal with umbilic points. 
\end{remark}

\begin{remark} We can actually estimate how small $\varepsilon > 0$ has to be such that Weyl's formula works. It suffices that
\begin{align*} \varepsilon < \frac{1}{\max_{p\in M}\{\max\{\lambda_1(p),\lambda_2(p)\}\}}.
\end{align*}
If the inequality above is true, then local parametrizations of the tubes that we used to prove the formula are injective. Indeed, in this case the points of $M_{\varepsilon}$ do not ``reach" the \emph{medial axis} of $M$ (see the proof of Theorem \ref{volumeestimate}).
\end{remark}

In the following corollary we will use the notion of mixed volumes of a convex body and the unit ball which are common in classical convexity and have two representations in the literature: as quermassintegrals and (used by us) \emph{intrinsic volumes}; see Section 3 in the survey \cite{Sang} and Section 4.2 in \cite{Schn}.

\begin{coro} Assume that $M$ is a compact surface without boundary which encloses a convex region, and recall that we denote by $B$ the unit ball of $(\mathbb{R}^3,||\cdot||)$. Also, denote by $V_3$ the usual volume given in $\mathbb{R}^3$ by the determinant. Then, in the Steiner formula
\begin{align} \label{steiner} V_3(M+\rho B) = \sum_{j=0}^3\rho^{n-j}\binom{n}{j}V_j(M,B),
\end{align}  
we have that, as in the Euclidean case, the intrinsic volume $V_2(M,B)$ equals the Minkowski surface area of $M$. Also, all of the intrinsic volumes $V_j(M,B)$ are given by quantities defined in terms of the geometry induced by the ambient norm $||\cdot||$ in $M$.
\end{coro}
\begin{proof} Procceeding as in the proof of Theorem \ref{weyl}, but parametrizing only the ``outer" portion of the tube, we get immediately
\begin{align*} V_3(M+\rho B) = V_3(M) + \rho\lambda_M(M) + \rho^2\int_MH\omega + \frac{\rho^3}{3}\int_MK\omega,
\end{align*}
where $\rho$ is sufficiently small. We emphasize that the intrinsic volumes are not just given by the geometry induced in $M$ by the norm, but that this happens in a way completely analogous to the Euclidean case. Also, notice that we are not normalizing the intrinsic volumes by the successive volumes of $n$-dimensional balls since this is not completely natural in the Minkowski case. 

\end{proof}
\begin{open} Is there any extension of all this for the case that the unit ball can be an arbitrary convex body no longer centered at the origin (and then called a \emph{gauge})? To answer that question and similar problems, it is certainly necessary to develop a theory called "differential geometry of gauges"; see, e.g., \cite{Gugg}.

\end{open}

It is worth mentioning that something similar holds for convex curves in a Minkowski plane $X$. In that case, the intrinsic volume $V_1$ equals the length in the \emph{anti-norm}, which is the dual norm induced in $X$ under the identification of $X$ and $X^*$ given by the fixed area form. We define it more precisely in Section \ref{total}, and we refer the reader to \cite{Ma-Swa} for more in this direction. 

\section{Total Minkowski curvature} \label{total}

It is well known that in classical differential geometry the total Gaussian curvature of a compact surface with positive Gaussian curvature equals the area of the unit sphere. We will extend this to the Minkowski context and, as a consequence, prove an analogue of Willmore's theorem. This theorem gives a lower bound for the total integral of the squared mean curvature of a compact, embedded surface without boundary. The optimal value characterizes round spheres. \\

We continue working with the area induced by Birkhoff-Gauss map (and the usual determinant in $\mathbb{R}^3$) as in (\ref{areaform}). In this case, the area of the unit sphere $\partial B$ is given by
\begin{align*} \lambda(\partial B) := \int_{\partial B}\omega.
\end{align*}
In what follows, we say that a surface is \emph{closed} whenever it is compact and without boundary. For the sake of simplicity, all of the immersions are assumed to be embeddings. 

\begin{teo}\label{totalgauss} Let $f:M\rightarrow(\mathbb{R}^3,||\cdot||)$ be a closed surface with positive Minkowski Gaussian curvature. Then
\begin{align*} \int_MK\omega = \lambda(\partial B).
\end{align*}
\end{teo}
\begin{proof} Since $M$ is compact and its Minkowski Gaussian curvature is positive, it follows that the Birkhoff-Gauss map $\eta:M\rightarrow\partial B$ is a diffeomorphism (see \cite[Theorem 3.2]{Ba-Ma-Tei1}). Hence
\begin{align*} \int_MK\omega = \int_M\mathrm{det}(d\eta)\cdot\omega = \int_{\partial B}\omega = \lambda(\partial B),
\end{align*}
where the second equality comes from the standard changing of variables formula, and from the fact that, for each $p \in M$, the Birkhoff normal is the same for $M$ at $p$ and $\partial B$ at $\eta(p)$

\end{proof}

\begin{remark} The same argument can be used to show that the total circular curvature of any simple, closed, strictly convex curve in $(\mathbb{R}^2,||\cdot||)$ equals the Minkowski length of the unit circle of the norm $||\cdot||$. Indeed, if $\gamma$ is such a curve, then we can assume that the map $t(s)$ defined in Remark \ref{curve} is a bijection from $[0,l(\gamma)]$ onto $[0,l(\partial B)]$, where $l$ stands for the Minkowski length, and hence
\begin{align*} \int_0^{l(\gamma)}k_c(s) \ ds = \int_0^{l(\gamma)}t'(s) \ ds = l(\partial B).
\end{align*}
\end{remark}

Clearly, Theorem \ref{totalgauss} can be extended to the case when $M$ has isolated points where the Minkowski Gaussian curvature vanishes. We will need this extension to prove the announced Willmore's theorem next. 

\begin{teo} If $f:M\rightarrow(\mathbb{R}^3,||\cdot||)$ is a closed surface whose Minkowski mean curvature is denoted by $H$, then
\begin{align*} \int_MH^2\omega \geq \lambda(\partial B),
\end{align*}
and equality holds if and only if $M$ is a round sphere. 
\end{teo}
\begin{proof} For each $p \in M$, denote $K^+(p) = \max\{K(p),0\}$. The image, under the Birkhoff-Gauss map, of the set of points in $M$ where the Minkowski Gaussian curvature is non-negative covers the whole of $\partial B$ (this is known in the Euclidean case, and the Minkowski case comes immediately as a consequence, since the Euclidean and Minkowski Gaussian curvatures always have the same sign). Hence, using the methods of the proof of Theorem \ref{totalgauss}, we get
\begin{align*} \int_MK^+\omega \geq \lambda(\partial B).
\end{align*}
From the arithmetic-geometric mean inequality we get immediately that $H^2 \geq |K|$, and then
\begin{align*} \int_MH^2\omega \geq \int_M|K|\omega \geq \int_MK^+\omega \geq \lambda(\partial B),
\end{align*}
which gives the inequality. We assume now that the equality holds. In this case, we must have $K \geq 0$, since the existence of a point where $K < 0$ would lead the second inequality above to be strict. We have
\begin{align*} \int_MH^2\omega = \int_MK\omega = \lambda(\partial B),
\end{align*}
and since $H^2 - K\geq 0$, by continuity we get that $H^2 = K$ at every point of $M$. Consequently, every point of $M$ is umbilic, and therefore $M$ is a Minkowski sphere (see \cite[Proposition 4.5]{Ba-Ma-Tei1}). The converse is obvious.

\end{proof}

Similarly to the Euclidean case, one can also bound the volume of the region enclosed by a compact, embedded surface without boundary in terms of its Minkowski mean curvature. 

\begin{teo}\label{volumeestimate} Let $M$ be a compact surface without boundary embedded in the Minkowski space $(\mathbb{R}^3,||\cdot||)$, whose Minkowski mean curvature is positive everywhere, and let $D$ be the volume of the region enclosed by it. Then we have the estimate
\begin{align*} \int_{M}\frac{1}{H}\cdot\omega \geq 3\mathrm{vol}(D),
\end{align*}
where $\mathrm{vol}(D)$ denotes the usual volume of the region $D$. Equality holds if and only if $M$ is a Minkowski sphere.
\end{teo} 
\begin{proof} Let $\eta$ be the (outward-pointing) Birkhoff normal of $M$, and for each $p \in M$, let $h(p)$ be the supremum of the distance that one can travel in the direction of $-\eta$ until a point $q$ is reached for which $p$ is no longer the unique point of $M$ such that $\mathrm{dist}(q,M) = ||p-q||$. We let
\begin{align*} D_{\varepsilon} := \{p - \varepsilon h(p)\eta:p \in M, \ \varepsilon \in [0,1]\}.
\end{align*}
The points of $D\setminus D_{\varepsilon}$ form the \emph{medial axis} of $M$. Adapting the proof of the Euclidean case, one can easily see that $\bar{D}_{\varepsilon} = D$, from which we get $\mathrm{vol}(D_{\varepsilon}) = \mathrm{vol}(D)$, since clearly $\mathrm{int}(\bar{D}_{\varepsilon}\setminus D_{\varepsilon}) = \emptyset$. Using local parametrizations of $M$, and a partition of the unity, we get, as in Theorem \ref{weyl} (but with a different sign, since our normal vector points outwards the parametrized region), the equality 
\begin{align*} \mathrm{vol}(D) = \mathrm{vol}(D_{\varepsilon}) = \int_M\left(\int_0^{h(p)}|(1-t\lambda_1)(1-t\lambda_2)| \ dt\right)\cdot \omega, 
\end{align*}
where, as before, $\lambda_1$ and $\lambda_2$ denote the principal Minkowski curvatures of $M$. Moreover, as in the Euclidean case, we have 
\begin{align*} \frac{1}{h(p)} \geq \max\{\lambda_1(p),\lambda_2(p)\} \ (\geq H).
\end{align*}
This can be directly checked by considering intersections of $M$ with normal planes and applying \cite[Theorem 8.3]{Ba-Ma-Sho}. Hence we get 
\begin{align*} 0 \leq (1-t\lambda_1)(1-t\lambda_2) = 1-2Ht+Kt^2 \leq 1-2Ht+H^2t^2 = (1-Ht)^2,
\end{align*}
for any $0 \leq t \leq h(p)$. Finally,
\begin{align*} \mathrm{vol}(D) = \int_M\left(\int_0^{h(p)}(1-t\lambda_1)(1-t\lambda_2) \ dt\right)\cdot \omega \leq \int_M\left(\int_0^{1/H}(1-Ht)^2 \ dt\right)\cdot \omega = \\  = \frac{1}{3}\int_M\frac{1}{H}\cdot \omega, 
\end{align*} 
as we wanted to prove. Clearly, equality holds if and only if $H^2 = K$. This characterizes Minkowski spheres.

\end{proof}

There is a two-dimensional version of Theorem \ref{volumeestimate}. To state and prove that, we need some preliminary theory. Let $(X,||\cdot||)$ be a normed plane, and endow $X$ with a nondegenerate symplectic form $\omega:X\times X\rightarrow\mathbb{R}$. This induces an area element and an orientation in $X$. Also, the natural identification of $X$ and its dual $X^*$ given by $\omega$ induces the so-called \emph{anti-norm} of $||\cdot||$ as
\begin{align*}||x||_a := \sup\{\omega(x,y):y \in S\},
\end{align*}
for any $x \in X$, where $S$ denotes the unit circle of $(X,||\cdot||)$. For more on this construction, we refer the reader to \cite{Ma-Swa}. In what follows, we still denote by $k_c$ the circular curvature \emph{in the norm}. 

\begin{teo}Let $\gamma(s)$ be a closed, smooth, simple curve parametrized by the arc-length in the anti-norm, and with everywhere positive circular curvature. Denoting by $\lambda(D)$ the area of the region $D$ enclosed by $\gamma$, we have the estimate
\begin{align*}2\lambda(D) \leq \int_0^{l_a(\gamma)}\frac{1}{k_c(s)} \ ds,
\end{align*}
where $l_a(\gamma)$ is the length of $\gamma$ measured in the anti-norm. Equality holds if and only if $\gamma$ is a circle of the norm $||\cdot||$. 
\end{teo}
\begin{proof} For simplicity, assume that $X$ is identified with $\mathbb{R}^2$, and that $\omega$ is the usual determinant in $\mathbb{R}^2$. Also, suppose that the orientation of $\gamma$ is counterclockwise. For each $s$, let $\eta(s)$ be the outward-pointing left Birkhoff normal unit vector to $\gamma$ at $\gamma(s)$, that is, $\eta(s) \dashv_B \gamma'(s)$ and $\omega(\eta,\gamma') = ||\eta||\cdot||\gamma'||_a = 1$ (see \cite{Ma-Swa}). Combining equation (4.5) in \cite{Ba-Ma-Sho} and Proposition 4.1 in the same paper, we get 
\begin{align}\label{eq1} \eta'(s) = k_c(s)\gamma'(s).
\end{align}
Similarly to what we have done in Theorem \ref{volumeestimate}, for each $s \in [0,l_a(\gamma)]$ let $h(s)$ be the supremum of the distance that one can travel in the direction of $-\eta$ until a point $q$ is reached for which $\gamma(s)$ is no longer the unique point of $\gamma$ such that $\mathrm{dist}(q,\gamma) = ||p-q||$. We let
\begin{align*} D_{\varepsilon} := \{p - \varepsilon h(p)\eta:p \in M, \ \varepsilon \in [0,1]\},
\end{align*}
and it is clear that $D\setminus D_{\varepsilon}$ is the medial axis of $\gamma$. Hence $\bar{D}_{\varepsilon} = D$. We parametrize $D_{\varepsilon}$ by $\gamma(s) -t\eta(s)$, where $(s,t) \in [0,l_a(\gamma)]\times[0,h(s)]$. Since $1/k_c(s)$ is the radius of an osculating Minkowski circle attached to $\gamma$ at $\gamma(s)$, we have the inequality 
\begin{align*} \frac{1}{h(s)} \geq k_c(s),
\end{align*}   
for any $s \in [0,l_a(\gamma)]$. As in Theorem \ref{volumeestimate}, this comes as a consequence of \cite[Theorem 8.3]{Ba-Ma-Sho}. Therefore,
\begin{align*} \lambda(D) = \lambda(D_{\varepsilon}) = \int_0^{l_a(\gamma)}\int_0^{h(s)}|\omega(\eta,\gamma'-t\eta')| \ dtds = \lambda(D_{\varepsilon}) = \int_0^{l_a(\gamma)}\int_0^{h(s)}1 - tk_c(s) \ dtds = \\ = \int_0^{l_a(\gamma)}\int_0^{1/k_c(s)}1-tk_c(s) \ dtds = \frac{1}{2}\int_0^{l_a(\gamma)}\frac{1}{k_c(s)} \ ds,
\end{align*} 
and this gives the desired inequality. Clearly, equality holds if and only if $h(s) = 1/k_c(s)$ for every $s \in [0,l_a(\gamma)]$, and this characterizes the Minkowski circles.

\end{proof}

We can also obtain an analogue of Huber's theorem, which states that a connected, oriented and complete surface in $\mathbb{R}^3$ with finite total Gaussian curvature has finite topological type (see \cite{white}). Our main ingredients are estimates for the total curvature given by the next proposition, which can be seen as a characterization of the Minkowski Gaussian curvature in terms of the Euclidean Gaussian curvature of the surface and of the unit sphere.

\begin{prop}\label{gaussratio} Let $f:M\rightarrow(\mathbb{R}^3,||\cdot||)$ be an immersed surface, and let $p \in M$. The Minkowski Gaussian curvature of $M$ at $p$ is given as the ratio
\begin{align*} K(p) = \frac{K_M(p)}{K_{\partial B}(\eta(p))},
\end{align*}
where $K_M(p)$ and $K_{\partial B}(\eta(p))$ denote the Euclidean Gaussian curvatures of $M$ and $\partial B$ at $p$ and $\eta(p)$, respectively. 
\end{prop}
\begin{proof} Recall that we denote by $u:\partial B_e\rightarrow \partial B$ the inverse of the (Euclidean) Gauss map of $\partial B$, and let $\xi:M\rightarrow\partial B$ be the Euclidean Gauss map of $M$. Then we get immediately
\begin{align*} \xi = \eta\circ u^{-1},
\end{align*}
and hence $d\xi = d\eta\circ du^{-1}$. This gives
\begin{align*} K(p) = \mathrm{det}(d\eta_{p}) = \frac{\mathrm{det}\left(d\xi_{\eta(p)}\right)}{\mathrm{det}\left(du^{-1}_{\eta(p)}\right)} = \frac{K_M(p)}{K_{\partial B}(\eta(p))},
\end{align*}
where we remind the reader that we always assume that $\partial B$ has positive Gaussian curvature at all of its points.

\end{proof}

\begin{remark} As in the case of Theorems \ref{main} and \ref{totalgauss}, this proposition is analogous to a result on the circular curvature of a curve in a normed plane: this curvature type can also be characterized as the ratio of Euclidean curvatures of the curve and of the unit circle (see equality (5.9) in \cite{Ba-Ma-Sho}). However, there is no characterization of this kind for the Minkowski principal and mean curvatures. 
\end{remark}

\begin{teo} Let $M$ be a connected, oriented, complete surface immersed in $(\mathbb{R}^3,||\cdot||)$, and let $K^-(p) = \min\{K(p),0\}$, $p \in M$. If 
\begin{align*} \int_M|K^-| \ \omega < \infty,
\end{align*}
then
\begin{align*} \int_MK^+ \omega < \infty,
\end{align*}
and there exist a compact two-dimensional manifold $N$ and points $p_1,\ldots,p_k \in N$ such that $M$ is homeomorphic to $N\setminus\{p_1,\ldots,p_k\}$. 
\end{teo}
\begin{proof} The classical Huber's theorem states that the result is true if the Minkowski Gaussian curvature $K$ is replaced by the usual Gaussian curvature $K_M$, and the Minkowski area element $\omega$ is replaced by the Euclidean area element $\omega_e$, which is defined as
\begin{align*} \omega_e(X,Y) = \mathrm{det}(X,Y,\xi),
\end{align*}
for each $p \in M$ and $X,Y \in T_pM$, where $\xi$ is the Euclidean Gauss normal of $M$ at $p$. Therefore, we just have to bound the Euclidean Gaussian curvature in terms of the Minkowski Gaussian curvature. Since it is clear that we have $\omega = \langle\eta,\xi\rangle\cdot\omega_e$, where $\langle\cdot,\cdot\rangle$ denotes the standard inner product of $\mathbb{R}^3$, we get
\begin{align*} \int_MK^+\omega = \int_M\frac{K^+_M}{K_{\partial B}}\langle\eta,\xi\rangle\cdot\omega_e,
\end{align*}
yielding the estimates
\begin{align*}\frac{\min\langle\eta,\xi\rangle}{\max(K_{\partial B})}\int_MK_M^+ \ \omega_e\leq \int_MK^+\omega\leq \frac{\max\langle\eta,\xi\rangle}{\min(K_{\partial B})}\int_MK_M^+ \ \omega_e,
\end{align*}
where the maxima and minima are taken over $\partial B$. Hence they are (positive) numbers which do not depend on $M$, but only on the geometry of $\partial B$. Similarly, we have the estimates
\begin{align*}\frac{\max\langle\eta,\xi\rangle}{\min(K_{\partial B})}\int_MK_M^- \ \omega_e\leq \int_MK^-\omega\leq \frac{\min\langle\eta,\xi\rangle}{\max(K_{\partial B})}\int_MK_M^- \ \omega_e.
\end{align*}
Combining these estimates with the classical version of Huber's theorem, we get the desired immediately.  

\end{proof}

\section{A version of Alexandrov's theorem}

Alexandrov's theorem states that any compact surface without boundary embedded in $\mathbb{R}^3$ with non-zero constant mean curvature is a sphere (see, e.g., \cite[Chapter 8]{li}). In this section we want to prove an analogous result for the Minkowski case: if $M$ is a compact surface without boundary embedded in $(\mathbb{R}^3,||\cdot||)$ whose Minkowski mean curvature is constant, then $M$ is a Minkowski sphere. First, we need some preliminary discussions.

Let $M$ be such a surface, and let $v:M\rightarrow\mathbb{R}^3$ be a smooth vector field. Let $\rho$ be given by the decomposition
\begin{align*} v(p) = T(p) + \rho(p)\eta(p), \ p \in M,
\end{align*}
where $T(p)\in T_pM$ is the tangential component of $v(p)$ and $\eta$ denotes the outward-pointing Birkhoff normal vector field of $M$. Let $\xi$ be the Euclidean normal vector field of $M$, and recall that we denote by $\langle\cdot,\cdot\rangle$ the usual inner product of $\mathbb{R}^3$. Taking the inner product on both sides of the equality above, we get
\begin{align*} \rho = \frac{\langle v,\xi\rangle}{\langle\eta,\xi\rangle}.
\end{align*}
With that, we can obtain a Minkowskian version of the divergence theorem which uses $\rho$ and the area element induced by the Birkhoff normal (instead of the Euclidean normal). To do so, recall that the Minkowski area element $\omega$ and the Euclidean area element $\omega_e$ are related by $\omega = \langle\eta,\xi\rangle\cdot\omega_e$. Assume that the vector field $v$ is defined in an open ball containing $M$, and denote by $D$ the region enclosed by $M$. We write
\begin{align*} \int_D\mathrm{div}(v) \ \mathrm{dvol} = \int_M\langle v,\xi\rangle\cdot \omega_e = \int_M\rho\cdot\langle\eta,\xi\rangle \cdot \omega_e = \int_M\rho\cdot\omega,
\end{align*}
where $\mathrm{dvol}$ denotes the usual volume form in $\mathbb{R}^3$, and the first equality is given by the classical divergence theorem. The equality above gives a Minkowskian version of the divergence theorem which is defined in terms of the Birkhoff normal vector field determined by the fixed norm in $\mathbb{R}^3$.

We turn now to the particular case where the vector field $v$ is simply the position vector of $\mathbb{R}^3$. In this case, $\mathrm{div}(v) = 3$ and we get the equality
\begin{align}\label{eq1} \int_M\rho\cdot\omega = 3\mathrm{vol}(D),
\end{align}
where $\rho$ is the \emph{affine distance function} from $M$ to the origin $o \in \mathbb{R}^3$. Indeed, $\rho$ is defined by the decomposition $p = p - o = T(p) + \rho(p)\eta(p)$, with $T(p) \in T_pM$. From \cite[III.9]{nomizu}, we have that $\rho$ satisfies the equality
\begin{align}\label{eq2} \int_M(1-\rho H)\cdot\omega = 0,
\end{align}
where the reader must be aware of the fact that our sign is changed because we assume that $\eta$ is outward-pointing.
\begin{teo}[Minkowskian version of Alexandrov's theorem] Let $M$ be a compact surface without boundary embedded in $(\mathbb{R}^3,||\cdot||)$. If $M$ has constant non-zero Minkowski mean curvature, then $M$ is a Minkowski sphere.
\end{teo}
\begin{proof} First of all, notice that there exists a point in $M$ where the Minkowski Gaussian curvature is positive (see \cite[Lemma 6.1]{Ba-Ma-Tei1}), and hence, since $\eta$ is outward-pointing, the Minkowski mean curvature is also positive at this point. It follows that if $H$ is constant, then $H$ is positive.
From (\ref{eq1}) and (\ref{eq2}) we have
\begin{align*} \int_M\omega = H\int_M\rho\cdot \omega = 3H\mathrm{vol}(D).
\end{align*}
Hence, we get the equality
\begin{align*} \int_M\frac{1}{H}\cdot\omega = 3\mathrm{vol}(D).
\end{align*}
By Theorem \ref{volumeestimate} it follows that $M$ is a Minkowski sphere. 

\end{proof}

\section{A Bertrand-Diguet-Puiseux theorem}

The classical Bertrand-Diguet-Puiseux theorem gives a characterization of the Gaussian curvature of a given point $p$ of a surface in terms of lengths and areas of small geodesic circles around $p$ (see \cite{spivak}). We can give a version of this theorem for the Minkowski Gaussian curvature, but with the disadvantage of using an auxiliary Euclidean structure on $\mathbb{R}^3$. Our main ingredient is Proposition \ref{gaussratio}. In what follows, the \emph{Euclidean geodesic circle} of center $p \in M$ and radius $r > 0$ (where $r > 0$ is small enough) is the set of points of $M$ whose distance (in the Euclidean ambient metric induced in $M$) to $p$ is less than or equal to $r$. 

\begin{teo} Let $p\in M$, and denote by $C_M(p,r)$ the length of the Euclidean geodesic circle centered at $p$ whose radius is $r> 0$. Also, write $C_{\partial B}(\eta(p),r)$ for the Euclidean geodesic circle of radius $r$ centered at $\eta(p)$. Then the Minkowski Gaussian curvature of $M$ at $p$ is given as
\begin{align*} K(p) = \lim_{\ r\rightarrow 0^+}\frac{2\pi r - C_M(p,r)}{2\pi r - C_{\partial B}(\eta(p))}.
\end{align*}
\end{teo}
\begin{proof} We prove the theorem using the classical Bertrand-Diguet-Puiseux theorem, which states that the Euclidean Gaussian curvature of $M$ at $p$ is given as
\begin{align*} K_M(p) = \frac{3}{\pi}\cdot\lim_{\ r\rightarrow 0^+}\frac{2\pi r - C_M(p,r)}{r^3}.
\end{align*}
In view of that and using Proposition \ref{gaussratio}, we get
\begin{align*} K(p) =  \frac{K_M(p)}{K_{\partial B}(\eta(p))} = \frac{\lim_{r\rightarrow 0^+}\frac{2\pi r - C_M(p,r)}{r^3}}{\lim_{r\rightarrow 0^+}\frac{2\pi r - C_{\partial B}(p,r)}{r^3}} = \lim_{\ r\rightarrow 0^+}\frac{2\pi r - C_M(p,r)}{2\pi r - C_{\partial B}(\eta(p))},
\end{align*}
where the last equality holds since we are assuming that $\partial B$ has non-zero Euclidean Gaussian curvature at every point. 

\end{proof}

\begin{remark} If $A_M(p,r)$ and $A_{\partial B}(\eta(p),r)$ denote the area of the geodesic circle with center $p$ and radius $r$ in $M$ and the area of the geodesic circle with center $\eta(p)$ and radius $r$ in $\partial B$, respectively, then the same argument can be applied to prove the equality
\begin{align*} K(p) = \lim_{\ r\rightarrow 0^+}\frac{\pi r^2 - A_M(p,r)}{\pi r^2-A_{\partial B}(\eta(p),r)}.
\end{align*}
\end{remark}

\begin{open} Can one obtain an analogue of the Bertrand-Diguet-Puiseux theorem which does not involve an auxiliary Euclidean structure?
\end{open}

\end{document}